\documentclass{amsart}
\usepackage{newtxtext,newtxmath}
\usepackage{miyamath,tikz-cd}
\DeclareMathOperator{\Hyp}{Hyp}
\DeclareMathOperator{\sH}{\mathscr{H}}
\DeclareMathOperator{\bC}{\mathbb{C}}

\DeclareMathOperator{\hbPV}{\widehat{\mathbb{P}^1_V}}
\DeclareMathOperator{\hbAV}{\widehat{\mathbb{A}^1_V}}
\DeclareMathOperator{\totimes}{\widetilde{\otimes}}
\DeclareMathOperator{\inv}{inv}
\DeclareMathOperator{\rmm}{m}
\DeclareMathOperator{\rqq}{\mathrm{q}}

\newcommand{\Ldagotimes}[1]{{\mathop{\otimes}\displaylimits^{\mathbb{L}}}^{\raisebox{-5pt}{\scriptsize\dag}}_{#1}}
\newcommand{\Ldagboxtimes}{{\mathop{\boxtimes}\displaylimits^{\mathbb{L}}}^{\raisebox{-5pt}{\scriptsize\dag}}}
\newcommand{\pdag}{{}^{\dag}}
\mathchardef\hyphen="2D
\title[Hypergeometric $\mathscr{D}$-modules under $p$-adic non-Liouvilleness condition]{Generalized hypergeometric arithmetic $\mathscr{D}$-modules under a $p$-adic non-Liouvilleness condition.}
\author{Kazuaki Miyatani}
\address{Department of Mathematics, Graduate School of Science, Hiroshima University. 1-3-1 Kagamiyama, Higashi-Hiroshima, 739-8526, Japan. }
\email{miyatani@hiroshima-u.ac.jp}
\urladdr{https://math.miyatani.org/}

\begin{document}

\begin{abstract}
    We prove that the arithmetic $\mathscr{D}$-modules associated with
    the $p$-adic generalized hypergeometric differential operators,
    under a $p$-adic non-Liouvilleness condition on parameters,
    are described as an iterative multiplicative convolution
    of (hypergeometric arithmetic) $\mathscr{D}$-modules of rank one.
    As a corollary, we prove the overholonomicity of hypergeometric arithmetic $\mathscr{D}$-modules
    under a $p$-adic non-Liouvilleness condition.
\end{abstract}

\maketitle

\section{Introduction.}
N.\ M.\ Katz \cite{Katz90ESDE} introduces the hypergeometric $\mathscr{D}$-modules $\sHyp(\balpha;\bbeta)$ over $\Gm$ with complex parameters
$\balpha=(\alpha_1,\dots,\alpha_m)\in\mathbb{C}^m$ and $\bbeta=(\beta_1,\dots,\beta_n)\in\mathbb{C}^n$,
and uses them to nicely describe an ``inductive'' structure of hypergeometric objects.
To be precise, Katz proves that $\sHyp(\balpha;\bbeta)$ is described as iterative 
multiplicative convolution, denoted by $\ast$, of the hypergeometric $\mathscr{D}$-modules $\sHyp(\alpha_i;\emptyset)$'s and $\sHyp(\emptyset;\beta_j)$'s
(i.e. those with only one parameter):

\begin{theorem*}[{\cite[5.3.1]{Katz90ESDE}}]
    Let $\balpha=(\alpha_1,\dots,\alpha_m)\in\mathbb{C}^m$ and $\bbeta=(\beta_1,\dots,\beta_n)\in\mathbb{C}^n$ be
    complex parameters, and assume that for any $i$ and $j$, $\alpha_i-\beta_j$ is not an integer.

    Then, there exist isomorphisms
    \begin{align*}
    \sHyp(\balpha; \bbeta)&\cong \sHyp(\alpha_1,\dots,\alpha_{m-1};\bbeta)\ast\sHyp(\alpha_m;\emptyset),\\
    \sHyp(\balpha; \bbeta)&\cong \sHyp(\balpha;\beta_1,\dots,\beta_{n-1})\ast\sHyp(\emptyset;\beta_n).\\
    \end{align*}
\end{theorem*}

In the previous article \cite{Miyatani}, the author introduces the $p$-adic hypergeometric differential operators
with $p$-adic parameters $\balpha\in(\mathbb{Z}_p)^m$ and $\bbeta\in(\mathbb{Z}_p)^n$
(under a choice of Dwork's $\pi$)
and also the arithmetic $\mathscr{D}$-modules $\sH_{\pi}(\balpha;\bbeta)$ associated with such differential operators.
The author then proves that, in the case where all components of $\balpha$ and $\bbeta$ lie in $\frac{1}{q-1}\mathbb{Z}$,
then $\sH_{\pi}(\balpha;\bbeta)$ has an analogous description as the theorem above
by using the multiplicative convolution of arithmetic $\mathscr{D}$-module on $\Gm$ \cite[3.2.5]{Miyatani}.
As an application of this theorem,
the author proves that a $p$-adic hypergeometric differential operator defines
an overconvergent $F$-isocrystal on $\bG_{\rmm}$ if $m\neq n$, and on $\bG_{\rmm}\setminus\{1\}$ if $m=n$
\cite[4.1.3]{Miyatani}.

\vspace{5mm}
The goal of this article is to extend this decomposition of
 hypergeometric arithmetic $\mathscr{D}$-modules $\sH_{\pi}(\balpha;\bbeta)$
to more general parameters which are not necessarily rational numbers (thus 
they do not necessarily come from a multiplicative character on the residue field).
In fact, we prove this under a $p$-adic non-Liouvilleness condition on parameters:

\begin{theorem*}[Theorem \ref{thm:maintheorem}]
    Let $\balpha=(\alpha_1,\dots,\alpha_m)\in(\mathbb{Z}_p)^m$ and $\bbeta=(\beta_1,\dots,\beta_n)\in(\mathbb{Z}_p)^n$ be
    parameters in $p$-adic integers.
Assume that, for any $i$ and $j$, 
$\alpha_i-\beta_j$ is not an integer nor a $p$-adic Liouville number.

Then, we have isomorphisms
\begin{align*}
\sH_{\pi}(\balpha; \bbeta)&\cong \sH_{\pi}(\alpha_1,\dots,\alpha_{m-1};\bbeta)\ast\sH_{\pi}(\alpha_m;\emptyset)[-1],\\
\sH_{\pi}(\balpha; \bbeta)&\cong \sH_{\pi}(\balpha;\beta_1,\dots,\beta_{n-1})\ast\sH_{\pi}(\emptyset;\beta_n)[-1].
\end{align*}
\end{theorem*}

Since an algebraic number in $\mathbb{Z}_p$ is not a $p$-adic Liouville number,
the theorem above is, in particular, applicable to any algebraic parameters with no integer differences.

As an application of this main theorem, we prove the quasi-$\Sigma$-unipotency in the sense of Caro \cite{Caro18MM},
in particular the overholonomicity,
of our $\sH_{\pi}(\balpha;\bbeta)$ under a stronger condition of $p$-adic non-Liouvilleness:

\begin{corollary*}[Proposition \ref{prop:overholonomic}]
    Let $\balpha=(\alpha_1,\dots,\alpha_m)\in(\mathbb{Z}_p)^m$ and $\bbeta=(\beta_1,\dots,\beta_n)\in(\mathbb{Z}_p)^n$
    be parameters in $p$-adic integers.
Assume that $(m,n)\neq (0,0)$, that $\alpha_i-\beta_j\not\in\mathbb{Z}$ for any $(i,j)$,
and that the subgroup $\Sigma$ of $\mathbb{Z}_p/\mathbb{Z}$ generated by $\alpha_i$'s and $\beta_j$'s
does not have a $p$-adic Liouville number.
Then, $\sH_{\pi}(\balpha;\bbeta)$ is a quasi-$\Sigma$-unipotent $\sD^{\dag}_{\hbPV,\bQ}$-module.
In particular, it is an overholonomic $\sD^{\dag}_{\hbPV,\bQ}$-module.
\end{corollary*}

Contrary to the results in the previous article,
our $\sH_{\pi}(\balpha;\bbeta)$'s do not necessarily
have a Frobenius structure
(in fact, for example, $\sH_{\pi}(\alpha;\emptyset)$ does not have a Frobenius structure
if $\alpha$ is not rational).
It is thus worth to remark that the corollary above gives examples of overholonomic $\sD^{\dag}$-modules
without assuming the existence of a Frobenius structure.

\vspace{5mm}
We conclude this introduction by explaining the organization of this article.

In Section 1, after a quick review of the theory of cohomological operations on arithmetic $\mathscr{D}$-modules,
we define the multiplicative convolution for arithmetic $\mathscr{D}$-modules
and study the relationship with Fourier transform.

In Section 2, we firstly introduce the hypergeometric arithmetic $\mathscr{D}$-modules.
Then, after recalling the notion of $p$-adic Liouvilleness,
we give a crucial lemma on hypergeometric arithmetic $\mathscr{D}$-modules
under a $p$-adic non-Liouvilleness condition,
which we will need in proving the main theorem.

In Section 3, we establish the main theorem and give an application to the quasi-$\Sigma$-unipotence.

\subsection*{Acknowledgement}
This work is supported by JSPS KAKENHI Grant Number 17K14170.

\subsection*{Conventions and Notations.}

In this article, $V$ denotes a complete discrete valuation ring
of mixed characteristic $(0,p)$
whose residue field $k$ is a finite field with $q$ elements.
The fraction field of $V$ is denoted by $K$.
We denote by $|\cdot|$ the norm on $K$ normalized by $|p|=p^{-1}$.
Throughout this article, we assume that there exists an element $\pi$ of $K$
that satisfies $\pi^{q-1}+(-p)^{(q-1)/(p-1)}=0$,
and fix such a $\pi$.

\section{Arithmetic $\mathscr{D}$-modules}
    \subsection{Cohomological operations on $D^{\rb}_{\coh}\big(\sD^{\dag}_{\sP,\bQ}(\pdag{T})\big)$'s.}
    In this subsection, we recall some notation and 
    fundamental properties concerning cohomological operations on $D^{\rb}_{\coh}\big(\sD^{\dag}_{\sP,\bQ}(\pdag{T})\big)$.

    \begin{definition}
        (i) A \emph{d-couple} is a pair $(\sP,T)$, where $\sP$ is a smooth formal scheme over $\Spf(V)$ and where $T$ is a divisor of the special fiber of $\sP$ (an empty set is also a divisor).
        If a $k$-variety $X$ is the special fiber of $(\sP, T)$, we say that
        $(\sP, T)$ \emph{realizes} $X$.

        (ii) A \emph{morphism of d-couples} $f\colon(\sP',T')\to(\sP,T)$ is
        a morphism $\overline{f}\colon\sP'\to\sP$ such that
        $\overline{f}(\sP'\setminus T')\subset\sP\setminus T$ and that
        $\overline{f}^{-1}(T)$ is a divisor (or empty).
        We say that $f$ \emph{realizes} the morphism $f_0\colon X'\to X$ of $k$-varieties
        if $(\sP', T')$ \resp{$(\sP, T)$} realizes $X'$ \resp{$X$} and
        if $f$ induces $f_0$.
    \end{definition}

    \begin{remark}
        In the previous article \cite{Miyatani}, we usually denote a morphism of d-couples by putting a tilde,
        like $\widetilde{f}$, and we use the notation $f$ for the morphism of $k$-varieties realized by $\widetilde{f}$.
        In this article, we do not put tildes on the name of a morphism of d-couples
        because we rarely need to write the name of the realized morphism of $k$-varieties.
    \end{remark}

    \begin{para}
    For each d-couple $(\sP, T)$, we denote by $\sO_{\sP,\bQ}(\pdag{T})$ the sheaf of
    functions on $\sP$ with overconvergent singularities along $T$ \cite[4.2.4]{Berthelot96},
    and denote by $\sD^{\dag}_{\sP,\bQ}(\pdag{T})$ the sheaf of
    differential operators on $\sP$ with overconvergent singularities along $T$ \cite[4.2.5]{Berthelot96}.
    \end{para}

    \begin{para}
        \emph{Extraordinary pull-back functors \cite[1.1.6]{Caro06CM}.}
        Let $f\colon (\sP',T')\to(\sP,T)$ be a morphism of d-couples.
        Then, we have the extraordinary pull-back functor
        \[
            f^!\colon D^{\rb}_{\coh}\big(\sD^{\dag}_{\sP,\bQ}(\pdag{T})\big)\longrightarrow D^{\rb}\big(\sD^{\dag}_{\sP',\bQ}(\pdag{T'})\big).
        \]

        If $\overline{f}$ is smooth, or if $\overline{f}$ induces an open immersion
        $\overline{f}^{-1}(\sP\setminus T)\hookrightarrow \sP'\setminus T'$,
        then the essential image of $f^!$ lies in $D^{\rb}_{\coh}\big(\sD^{\dag}_{\sP',\bQ}(\pdag{T'})\big)$.
        In the case where $\overline{f}$ is an isomorphism, we also denote $f^!$ by $f^{\ast}$.
        In this case, we have $f^{\ast}(\sM)=\sD^{\dag}_{\sP,\bQ}(\pdag{T'})\otimes_{\sD^{\dag}_{\sP',\bQ}(\pdag{f^{-1}(T)})}\overline{f}^{-1}(\sM)$.

        Let $f'\colon (\sP'',T'')\to (\sP',T')$ be another morphism of d-couples
        and let $\sM$ be an object of $D^{\rb}_{\coh}\big(\sD^{\dag}_{\sP,\bQ}(\pdag{T})\big)$.
        Then, as long as $f^!\sM$ belongs to $D^{\rb}_{\coh}\big(\sD^{\dag}_{\sP',\bQ}(\pdag{T'})\big)$,
        we have a natural isomorphism $f'^!f^!(\sM)=(f\circ f')^!(\sM)$ of functors.
    \end{para}

    \begin{para}
    \emph{Ordinary push-forward functors \cite[1.1.6]{Caro06CM}.}
    Let $f\colon (\sP',T')\to(\sP,T)$ be a morphism of d-couples.
    Then, we have a push-forward functor
    \[
        f_{+}\colon D^{\rb}_{\coh}\big(\sD^{\dag}_{\sP',\bQ}(\pdag{T'})\big)\to D^{\rb}\big(\sD^{\dag}_{\sP,\bQ}(\pdag{T})\big).
    \]

    If $\overline{f}$ is proper and if $T'=\overline{f}^{-1}(T)$, then the essential image
    of $f_{+}$ lies in $D^{\rb}_{\coh}\big(\sD^{\dag}_{\sP,\bQ}(\pdag{T})\big)$.

    Let $f'\colon (\sP'',T'')\to (\sP',T')$ be another morphism of d-couples,
    and let $\sM''$ be an object of $D^{\rb}_{\coh}\big(\sD^{\dag}_{\sP'',\bQ}(\pdag{T''})\big)$.
    Then, as long as $f'_{+}\sM''$ is an object of $D^{\rb}_{\coh}\big(\sD^{\dag}_{\sP',\bQ}(\pdag{T'})\big)$,
    we have a natural isomorphism $f_+f'_{+}\sM''\cong(f\circ f')_{+}\sM''$.
    
    If $\sP'=\sP$ and if $\overline{f}$ is the identity morphism on $\sP$
    (thus $f$ represents an open immersion),
    then $f_{+}$ is obtained by considering the complex of $\sD^{\dag}_{\sP,\bQ}(\pdag{T'})$-modules
    as a complex of $\sD^{\dag}_{\sP,\bQ}(\pdag{T})$-module
    via the inclusion $\sD^{\dag}_{\sP,\bQ}(\pdag{T})\hookrightarrow\sD^{\dag}_{\sP,\bQ}(\pdag{T'})$.

    The base change is also available.
    Suppose that we are given a cartesian diagram of d-couples
    \[
    \begin{tikzcd}
        {(\mathscr{Q}', D')} \arrow[r, "g'"] \arrow[d, "f’"'] & {(\mathscr{P}', T')} \arrow[d, "f"] \\
        {(\mathscr{Q}, D)} \arrow[r, "g"'] & {(\mathscr{P}, T),}
    \end{tikzcd}
    \]
    and let $\sM$ be an object of $D^{\rb}_{\coh}\big(\sD^{\dag}_{\mathscr{Q},\bQ}(\pdag{D})\big)$.
    If $f'^!\sM$ belongs to $D^{\rb}_{\coh}\big(\sD^{\dag}_{\mathscr{Q}',\bQ}(\pdag{D'})\big)$
    and if $g_+\sM$ belongs to $D^{\rb}_{\coh}\big(\sD^{\dag}_{\mathscr{P},\bQ}(\pdag{T})\big)$,
    then we have a natural isomorphism $g'_+f'^!\sM\cong f^!g_+\sM$ by \cite[Remark in 5.7]{Abe14}.
    \end{para}

    \begin{para}
    \emph{Interior tensor functor.}
    Let $(\sP, T)$ be a d-couple.
    Then, we have an overconvergent tensor functor \cite[2.1.3]{Caro2015}
    \[
        \Ldagotimes{\sO_{\sP,\bQ}(\pdag{T})}\colon D^{\rb}_{\coh}\big(\sD^{\dag}_{\sP,\bQ}(\pdag{T})\big)\times D^{\rb}_{\coh}\big(\sD^{\dag}_{\sP,\bQ}(\pdag{T})\big)
        \to D^{\rb}\big(\sD^{\dag}_{\sP,\bQ}(\pdag{T})\big).
    \]
    We define an interior tensor functor
    \[
        \widetilde{\otimes}_{(\sP,T)}\colon D^{\rb}_{\coh}\big(\sD^{\dag}_{\sP,\bQ}(\pdag{T})\big)\times D^{\rb}_{\coh}\big(\sD^{\dag}_{\sP,\bQ}(\pdag{T})\big)
        \to D^{\rb}\big(\sD^{\dag}_{\sP,\bQ}(\pdag{T})\big)
    \]
    by $\sM\widetilde{\otimes}_{(\sP,T)}\sN\defeq\sM\displaystyle\Ldagotimes{\sO_{\sP,\bQ}(\pdag{T})}\sN[-\dim\sP]$.
    If no confusion would occur, we omit the subscript $(\sP,T)$.

    Let $f\colon(\sP',T')\to(\sP,T)$ be a morphism of d-couples, 
    and let $\sM$ and $\sN$ be objects of $D^{\rb}_{\coh}\big(\sD^{\dag}_{\sP,\bQ}(\pdag{T})\big)$.
    Assume that $\sM\totimes_{(\sP,T)}\sN$ belongs to $D^{\rb}_{\coh}\big(\sD^{\dag}_{\sP,\bQ}(\pdag{T})\big)$
    and that $f^!\sM$ and $f^!\sN$ belong to $D^{\rb}_{\coh}\big(\sD^{\dag}_{\sP',\bQ}(\pdag{T'}\big)\big)$.
    Then, we have an isomorphism 
    $f^!(\sM\widetilde{\otimes}_{(\sP,T)}\sN)\cong (f^!\sM)\widetilde{\otimes}_{(\sP',T')}(f^!\sN)$.
    by \cite[(2.1.9.1)]{Caro2015}.

    The projection formula is also available.
    Namely, let $f\colon(\sP',T')\to(\sP,T)$ be a morphism of d-couples,
    let $\sM$ be an object of $D^{\rb}_{\coh}\big(\sD^{\dag}_{\sP',\bQ}(\pdag{T'})\big)$,
    and let $\sN$ be an object of $D^{\rb}_{\coh}\big(\sD^{\dag}_{\sP,\bQ}(\pdag{T})\big)$.
    Assume that $f^!\sN$, $\sM\totimes_{(\sP',T')}f^!\sN$ and $f_+\sM$ are all coherent objects.
    Then, we have an isomorphism $f_+(\sM\totimes_{(\sP',T')}f^!\sN)\cong (f_+\sM)\totimes_{(\sP,T)}\sN$
    by \cite[2.1.6]{Caro2015}.
    \end{para}

    \begin{para}
    \emph{Exterior tensor functors \cite[2.3.3]{Caro2015}.}
    At last, we discuss the exterior tensor functor.
    Let $(\sP_1, T_1)$ and $(\sP_2, T_2)$ be two d-couples,
    and let $(\sP, T)\defeq(\sP_1,T_1)\times(\sP_2,T_2)$ be the product of them,
    that is, $\sP\defeq\sP_1\times\sP_2$ and $T\defeq(T_1\times P_2)\cup(P_1\times T_2)$.
    Then, we have an exterior tensor functor
    \[
        \Ldagboxtimes\colon D^{\rb}_{\coh}\big(\sD^{\dag}_{\sP_1,\bQ}(\pdag{T_1})\big)\times D^{\rb}_{\coh}\big(\sD^{\dag}_{\sP_2,\bQ}(\pdag{T_2})\big)
        \to D^{\rb}_{\coh}\big(\sD^{\dag}_{\sP,\bQ}(\pdag{T})\big).
    \]
    
    As usual, this functor can be described as follows.
    In the situation above, let $\pr_i\colon (\sP, T)\to(\sP_i, T_i)$ be projections for $i=1,2$.
    Then, we have an isomorphism \cite[(2.3.5.2)]{Caro2015}
    \[
      \sE\Ldagboxtimes\sF\cong \pr_1^!\sE\totimes_{(\sP,T)}\pr_2^!\sF.
    \]

    The K\"unneth formula is also available for this exterior tensor functor \cite[(2.3.7.2)]{Caro2015}.
    \end{para}

\subsection{Fourier transform.}
  In this subsection, we recall the notion of Fourier transform for arithmetic $\mathscr{D}$-modules \cite{Huyghe04}.

  \begin{para}
  Recall from Conventions and Notations that, 
  in this article, we fix an element $\pi$ in $K$ that satisfies $\pi^{q-1}+(-p)^{(q-1)/(p-1)}=0$.
  Let $\sL_{\pi}$ denote the Dwork isocrystal associated with $\pi$.
  \end{para}

  \begin{para}
    Let us introduce notations which we need to define the Fourier transform.
\[
    p_1, p_2\colon (\sP,T)\defeq \big(\widehat{\bP^1_V}, \{\infty\}\big) \times\big(\widehat{\bP^1_V}, \{\infty\}\big)
    \rightrightarrows \big(\widehat{\bP^1_V}, \{\infty\}\big)
\]
be the first and the second projection, respectively.
There exists a smooth formal scheme $\widetilde{\sP}$ and a projective morphism $\overline{f}\colon\widetilde{\sP}\to\sP=\hbPV\times\hbPV$
such that $\overline{f}$ induces an isomorphism $\overline{f}^{-1}\left(\hbAV\times\hbAV\right)\cong\hbAV\times\hbAV$
and that this isomorphism followed by the multiplication map $\hbAV\times\hbAV\to\hbAV$
extends to a morphism $\overline{\lambda}\colon\widetilde{\sP}\to\hbPV$.
Then, $\overline{f}$ \resp{$\overline{\lambda}$} defines the morphism of d-couples
$f\colon \big(\widetilde{\sP}, \overline{f}^{-1}(T)\big)\to (\sP, T)$
\resp{$\lambda\colon \big(\widetilde{\sP}, \overline{f}^{-1}(T)\big)\to\big(\hbPV, \{\infty\}\big)$}.
Finally, we put $\sN_{\pi}\defeq f_+\lambda^!(\sL_{\pi}[-1])$.
Because $\sL_{\pi}$ is an overconvergent isocrystal, $\sN_{\pi}$ is an object of $D^{\rb}_{\coh}\big(\sD^{\dag}_{\widehat{\bP^1_V},\bQ}(\pdag{\{\infty\}})\big)$.
\label{para:notationfourier}
    \end{para}

\begin{definition}
  The functor
  \[
    \FT_{\pi}\colon D^{\rb}_{\coh}\big(\sD^{\dag}_{\widehat{\bP^1_V},\bQ}(\pdag{\{\infty\}})\big)
    \longrightarrow D^{\rb}_{\coh}\big(\sD^{\dag}_{\widehat{\bP^1_V},\bQ}(\pdag{\{\infty\}})\big)
  \]
  is defined by sending $\sM$ in $D^{\rb}_{\coh}\big(\sD^{\dag}_{\widehat{\bP^1_V},\bQ}(\pdag{\{\infty\}})\big)$ to
  \[
    \FT_{\pi}(\sM) = p_{2,+}\big(p_1^!\sM\totimes_{(\sP,T)}\sN_{\pi}\big).
  \]
  This object $\FT_\pi(\sM)$ is called the \emph{geometric Fourier transform} of $\sM$.
\end{definition}

\begin{remark}
    It is a central result of \cite{Huyghe04} that $\FT_{\pi}$ sends $D^{\rb}_{\coh}\big(\sD^{\dag}_{\hbPV,\bQ}(\pdag{\{\infty\}})\big)$.

    The argument in loc.\ cit.\ also shows that, if $\sM$ belongs to $D^{\rb}_{\coh}\big(\sD^{\dag}_{\hbPV,\bQ}(\pdag{\{\infty\}})\big)$,
    then $p_1^!\sM\totimes_{(\sP,T)}\sN_{\pi}$ is also an object of $D^{\rb}_{\coh}\big(\sD^{\dag}_{\sP,\bQ}(\pdag{T})\big)$. 
    In fact, we may assume that $\sM$ is a (single) coherent $\sD^{\dag}_{\hbPV,\bQ}(\pdag{\{\infty\}})$-module placed at degree zero,
    and since such a coherent module has a free resolution \cite[5.3.3, (ii)]{Huyghe98},
    we may assume that $\sM=\sD^{\dag}_{\hbPV,\bQ}(\pdag{\{\infty\}})$.
    The claim follows from the calculation in \cite[4.2.2]{Huyghe04}.
    \label{rem:cohoffourier}
\end{remark}

\begin{para}
The geometric Fourier transform has another important description after passing to the global sections.
Let $A_1(K)^{\dag}$ be the ring defined by
\[
    A_1(K)^{\dag}\defeq\bigg\{\sum_{l,k\in\mathbb{N}}a_{l,k}x^l\partial^{[k]} \,\bigg|\, a_{l,k}\in K, \exists C>0, \exists \eta<1, |a_{l,k}|_p<C\eta^{l+k}\bigg\}.
\]
Then, by the $\mathscr{D}^{\dag}$-affinity \cite[5.3.3]{Huyghe98},
the functor $\Gamma\big(\widehat{\bP^1_V}, -\big)$ on the category of coherent $\mathscr{D}^{\dag}_{\widehat{\bP^1_V},\bQ}(\pdag{\{\infty\}})$-modules
is exact and gives an equivalence of this category with the category of coherent $A_1(K)^{\dag}$-modules (cf. \cite[p.915]{Huyghe98}).
Under this identification, the geometric Fourier transform is described as follows.
\end{para}

\begin{proposition}[{\cite[5.3.1]{Huyghe04}}]
Let $\varphi_{\pi}\colon A_1(K)^{\dag}\to A_1(K)^{\dag}$ be the ring automorphism defined by
$\varphi_{\pi}(x)=-\partial/\pi$ and $\varphi_{\pi}(\partial)=\pi x$.
Let $\sM$ be a coherent $A_1(K)^{\dag}$-module and denote by $\varphi_{\pi,\ast}\sM$
the coherent $A_1(K)^{\dag}$-module obtained by letting $A_1(K)^{\dag}$ act on $\sM$
 via $\varphi_{\pi}$.
Then, we have a natural isomorphism $\FT_{\pi}(\sM)\cong\varphi_{\pi,\ast}\sM[-1]$.
\label{prop:arithfourier}
\end{proposition}

\subsection{Multiplicative Convolutions}

In this subsection, we define the notion of multiplicative convolution
and study how it is related with Fourier transform.

\begin{para}
We follow the notation in the previous subsection.
We put $(\sP,T')\defeq\big(\hbPV, \{0,\infty\}\big)\times\big(\hbPV,\{0,\infty\}\big)$,
namely, $\sP=\hbPV\times\hbPV$ (which is compatible with the notation in \ref{para:notationfourier})
and $T'\defeq\big(\{0,\infty\}\times\bP^1_k\big)\cup\big(\bP^1_k\times\{0,\infty\}\big)$.
Let $\pr_1, \pr_2\colon (\sP,T') \rightrightarrows \big(\widehat{\bP^1_V}, \{0, \infty\}\big)$
denote the first and the second projection, respectively.
We denote by $f'\colon\big(\widetilde{\sP}, \overline{f}^{-1}(T')\big)\to(\sP, T')$
\resp{$\lambda'\colon\big(\widetilde{\sP}, \overline{f}^{-1}(T')\big)\to\big(\hbPV, \{0,\infty\}\big)$}
the morphism of d-couples defined by $\overline{f}$ \resp{$\overline{\lambda}$}.
\end{para}

\begin{definition}
    We define a \emph{multiplicative convolution} functor
    \[
        \ast\colon D^{\rb}_{\coh}\big(\sD^{\dag}_{\hbPV,\bQ}(\pdag{\{0,\infty\}})\big)\times D^{\rb}_{\coh}\big(\sD^{\dag}_{\hbPV,\bQ}(\pdag{\{0,\infty\}})\big)
        \longrightarrow
        D^{\rb}\big(\sD^{\dag}_{\hbPV,\bQ}(\pdag{\{0,\infty\}})\big)
    \]
    by $\displaystyle\sE\ast\sF\defeq \lambda'_{+}f'^!\big(\sE\Ldagboxtimes\sF\big)=\lambda'_+f'^!\big(\pr_1^!\sE\totimes_{(\sP,T')}\pr_2^!\sF\big)$.
\end{definition}

  In the following, we let $\inv\colon\big(\hbPV, \{0,\infty\}\big)\to\big(\hbPV,\{0,\infty\}\big)$ denote
  the morphism of d-couples defined by $\overline{\inv}\colon\hbPV\to\hbPV; x\mapsto x^{-1}$.

\begin{lemma}
  Let $\sE$ be an object of $D^{\rb}_{\coh}\big(\sD^{\dag}_{\hbPV,\bQ}(\pdag{\{0,\infty\}})\big)$
  and let $\sF$ be an overconvergent isocrystal on $\mathbb{G}_{\rmm,k}$
  considered as an object of $D^{\rb}_{\coh}\big(\sD^{\dag}_{\hbPV,\bQ}(\pdag{\{0,\infty\}})\big)$.
  Then, we have a natural isomorphism
  \[
  \sE\ast\sF\cong \pr_{2,+}\left(\pr_1^!\inv^\ast\sE\totimes_{(\sP,T')} f'_{+}\lambda'^!\mathscr{F}\right)
  \]
  in $D^{\rb}\big(\sD^{\dag}_{\hbPV,\bQ}(\pdag{\{0,\infty\}})\big)$.
  \label{lem:symmetricgeneralization}
\end{lemma}

\begin{proof}
  Let $\sigma\colon\big(\widetilde{\sP}, \overline{f}^{-1}(T')\big)\to(\sP,T')$ denote the morphism
  defined by $\overline{\sigma}=(\overline{\inv}\circ\overline{\pr_1}\circ\overline{f}, \overline{\lambda})$.
  Note that $\sigma$ represents the isomorphism $\bG_{\rmm,k}\times\bG_{\rmm,k}\to\bG_{\rmm,k}\times\bG_{\rmm,k}; (x,y)\mapsto (x^{-1}, xy)$.
  Since $\lambda'=\pr_{2}\circ \sigma$ and since $\sigma_+$ preserves coherence,
  we have an identification $\lambda'_+=\pr_{2,+}\circ\sigma_+$.
  By using this fact, we have
\[
  \sE\ast\sF=\lambda'_+f'^!\left(\pr_1^!\sE\totimes_{(\sP,T')}\pr_2^!\sF\right)
  \cong \pr_{2,+}\sigma_{+}\left(f'^!\pr_1^!\sE\totimes_{(\widetilde{\sP},\overline{f}^{-1}(T'))} f'^!\pr_2^!\sF\right).
\]
Moreover, since $\pr_1\circ f'=\inv\circ\pr_1\circ\sigma$, and since each of $f'^!$, $\inv^!$ and $\pr_1^!$ preserves coherence,
we have an identification $f'^!\pr_1^!=\sigma^!\pr_1^!\inv^\ast$ and therefore
\begin{align*}
  \pr_{2,+}\sigma_{+}\left(f'^!\pr_1^!\sE\totimes f'^!\pr_2^!\sF\right)
  &\cong \pr_{2,+}\sigma_{+}\left(\sigma^!\pr_1^!\inv^\ast\sE\totimes f'^!\pr_2^!\sF\right)\\
  &\cong \pr_{2,+}\left(\pr_1^!\inv^\ast\sE\totimes\sigma_{+}f'^!\pr_2^!\sF\right).
\end{align*}
Since $\sigma$ represents an involution on $\bG_{\rmm,k}\times\bG_{\rmm,k}$
and since $\sF$ is an overconvergent isocrystal,
we have $\sigma_{+}f'^!\pr_2^!\sF=f'_+\sigma^!\pr_2^!\sF=f'_+\lambda'^!\sF$, which completes the proof.
\end{proof}

\begin{proposition}
    \label{prop:fourierandconv}
    We denote by $j\colon \big(\widehat{\bP^1_V},\{0,\infty\}\big)\to \big(\widehat{\bP^1_V},\{\infty\}\big)$ the morphism of d-couples such that
    $\overline{j}=\id_{\hbPV}$. \textup{(}Thus, $j$ realizes the inclusion $\bG_{\rmm,k}\hookrightarrow\bA^1_k$.\textup{)}
    Let $\sM$ be an object of $D^{\rb}_{\coh}\big(\sD^{\dag}_{\widehat{\bP^1_V},\bQ}(\pdag{\{0,\infty\}})\big)$,
    and assume that $j_+\inv^{\ast}\sM$ belongs to $D^{\rb}_{\coh}\big(\sD^{\dag}_{\widehat{\bP^1_V},\bQ}(\pdag{\{\infty\}})\big)$.
    Then, we have a natural isomorphism
    \begin{equation}
        j^{\ast}\big(\FT_{\pi}(j_{+}\inv^{\ast}\sM)\big) \cong \sM\ast (j^{\ast}\sL_{\pi})[-1]
        \label{eq:fourierandconv}
    \end{equation}
\end{proposition}
\begin{proof}
    Put $(\sP, T_A)\defeq \big(\widehat{\bP^1_V}, \{\infty\}\big)\times\big(\widehat{\bP^1_V},\{0,\infty\}\big)$.
    Let $\pr_{1,A}\colon  (\sP,T_A) \to\big(\widehat{\bP^1_V},\{\infty\}\big)$ \resp{$\pr_{2,A}\colon (\sP,T_A)\to\big(\widehat{\bP^1_V},\{0,\infty\})$,
        $j_A\colon  (\sP,T')\to (\sP,T_A)$}
    be the morphisms of d-couples defined by the first projection \resp{the second projection, the identity morphism} on $\sP=\widehat{\bP^1_V}\times\widehat{\bP^1_V}$.
    This morphism represents the first projection $\mathbb{A}^1_k\times\bG_{\rmm,k}\to\bA^1_k$
    \resp{the second projection $\bA^1_k\times\bG_{\rmm,k}\to\bG_{\rmm,k}$,
    and the inclusion $\bG_{\rmm,k}\times\bG_{\rmm,k}\hookrightarrow\bA^1_k\times\bG_{\rmm,k}$}.

    Then, the definition of Fourier transform, we obtain a natural identification
    \[
        j^{\ast}\big(\FT_{\pi}({j}_+\inv^{\ast}\sM)\big)
        = \pr_{2,A,+}\big({\pr_{1,A}}^!{j}_+\inv^{\ast}\sM\totimes j''^{\ast}f_+\lambda^!\sL_{\pi}\big)[-1],
    \]
    where $j''\colon(\sP,T_A)\to(\sP,T)$ is the morphism of d-couples defined by $\overline{j''}=\id_{\sP}$,
    thus represents the inclusion $\bA^1_k\times\bG_{\rmm,k}\hookrightarrow\bA^1_k\times\bA^1_k$.
    Here, in the right-hand side, by the coherence assumption and Remark \ref{rem:cohoffourier},
    \[
      \pr_{1,A}^!j_+\inv^{\ast}\sM\totimes j''^{\ast}f_+\lambda^!\sL_{\pi}
      \cong j''^{\ast}\big(p_1^!j_+\inv^{\ast}\sM\totimes f_+\lambda^!\sL_{\pi}\big)
    \]
    belongs to $D^{\rb}_{\coh}\big(\sD^{\dag}_{\sP,\bQ}(\pdag{T_A})\big)$.

    Now, again by the coherence assumption, we have a base change isomorphism
    \[
        \pr_{1,A}^!j_+\inv^{\ast}\sM\cong j_{A,+}\pr_1^!\inv^\ast\sM.
    \]
    Moreover, since $j^!_{A,+}j''^!f_+\lambda^!\sL_{\pi}\cong f'_+\lambda'^!j^{\ast}\sL_{\pi}$,
    we see that
    \begin{align*}
        \pr_{1,A}^!j_+\inv^{\ast}\sM\totimes j''^\ast f_+\lambda^!\sL_{\pi}
        & \cong j_{A,+}\pr_{1}^!\,\inv^{\ast}\sM\totimes j''^{\ast}f_+\lambda^!\sL_{\pi}\\
        &\cong j_{A,+}\big(\pr_1^!\,\inv^{\ast}\sM\totimes j_{A,+}^!j''^!f_+\lambda^!\sL_{\pi}\big) \\
        & \cong
        j_{A,+}\left(\pr_{1}^!\inv^{\ast}\sM\totimes f'_+\lambda'^!j^{\ast}\sL_{\pi}\right).
    \end{align*}
    Since this object belongs to $D^{\rb}_{\coh}\big(\sD^{\dag}_{\sP,\bQ}(\pdag{T_A})\big)$ and $\pr_2=\pr_{2,A}\circ j_A$, we see that
    \[
        \pr_{2,A,+}j_{A,+}\left(\pr_1^!\inv^{\ast}\sM\totimes f'_+\lambda'^!j^{\ast}\sL_{\psi}\right)
        \cong
        \pr_{2,+}\left(\pr_1^!\inv^{\ast}\sM\totimes f'_+\lambda'^!j^{\ast}\sL_{\psi}\right).
    \]
    By Lemma \ref{lem:symmetricgeneralization}, this is isomorphic to the right-hand side of (\ref{eq:fourierandconv}) as desired.
\end{proof}

\section{Hypergeometric arithmetic $\mathscr{D}$-modules.}

\subsection{Definitions and fundamental properties.}

\begin{para}
Firstly, let us define a hypergeometric arithmetic $\sD$-module on $\bG_{\rmm,k}$
as a coherent $\sD^{\dag}_{\hbPV,\bQ}(\pdag{\{0,\infty\}})$-module.
Note that the category of coherent $\sD^{\dag}_{\hbPV,\bQ}(\pdag{\{0,\infty\}})$-modules
is identified with the category of coherent $B_1(K)^{\dag}$-modules \cite[5.3.3 and p.915]{Huyghe98},
where
\[
  B_1(K)^{\dag}\defeq\bigg\{\sum_{l\in\mathbb{Z},k\in\mathbb{N}}a_{l,k}x^l\partial^{[k]} \,\bigg|\, a_{l,k}\in K, \exists C>0, \exists \eta<1, |a_{l,k}|<C\eta^{\max(l,-l)+k}\bigg\}.
\]
\end{para}

\begin{definition}
    \label{def:hgoperators}
    Let $\alpha_1,\ldots,\alpha_m,\beta_1,\ldots,\beta_n$ be elements of $K$.
    We write the sequence $\alpha_1,\ldots,\alpha_m$ by $\balpha$ and $\beta_1,\ldots,\beta_n$ by $\bbeta$.
\begin{itemize}
    \item[\textup{(i)}] We define the hypergeometric operator $\Hyp_{\pi}(\balpha;\bbeta)=\Hyp_{\pi}(\alpha_1,\ldots,\alpha_m;\beta_1,\ldots,\beta_n)$ to be
    \[
        \Hyp_{\pi}(\balpha;\bbeta) \defeq \prod_{i=1}^m (x\partial-\alpha_i) - (-1)^{m+np}\pi^{m-n}x\prod_{j=1}^n(x\partial-\beta_j)
    \]

\item[\textup{(ii)}] We define a $B_1(K)^{\dag}$-module $\sH_{\pi}(\balpha;\bbeta)=\sH_{\pi}(\alpha_1,\ldots,\alpha_m;\beta_1,\ldots,\beta_n)$ by
\[
    \sH_{\pi}(\balpha;\bbeta) \defeq B_1(K)^{\dag}/B_1(K)^{\dag}\Hyp_{\pi}(\balpha;\bbeta).
\]
This is also considered as an object of $D^{\rb}_{\coh}\big(\sD^{\dag}_{\hbPV,\bQ}(\pdag{\{0,\infty\}})\big)$
by putting it on degree zero.
\end{itemize}
\end{definition}

\begin{remark}
    By definition, $\sH_{\pi}(\emptyset; \emptyset)$ is the delta module at $1$.

    If $(m,n)=(1,0)$, we may immediately check the isomorphism
    $\sH_{\pi}(\alpha;\emptyset)\cong j^{\ast}\sL_{\pi}\otimes^{\dag}\sK_{\alpha}$,
    where $\sK_{\alpha}$ is the Kummer isocrystal associated with $\alpha$.
    Similarly, if $(m,n)=(1,0)$, we get $\sH_{\pi}(\emptyset;\beta)\cong\inv^{\ast}\big(j^{\ast}\sL_{(-1)^p\pi}\otimes^{\dag}\sK_{-\beta})$.
(Recall that $\inv\colon(\hbPV, \{0,\infty\})\to(\hbPV,\{0,\infty\})$ denotes the morphism of d-couples
defined by $\overline{\inv}\colon x\mapsto x^{-1}$.)
\label{rem:m+n=1}
\end{remark}

\begin{para}
The goal of this article is to prove, under a $p$-adic non-Liouville condition,
 that $\sH_{\pi}(\balpha;\bbeta)$ can be obtained inductively
in terms of multiplicative convolution.
\end{para}

\begin{para}
The following lemma is obtained by a straight-forward calculation as in \cite[Lemma 3.1.3]{Miyatani}.
(In loc. cit., (ii) is stated in the case where $\gamma\in\frac{1}{q-1}\mathbb{Z}$, but this condition is not necessary.)
\end{para}

\begin{lemma}[{\cite[Lemma 3.1.3]{Miyatani}}]
    \label{lem:calcofhyp}
    Under the notation in Definition \ref{def:hgoperators}, $\sH_{\pi}(\balpha; \bbeta)$ has the following properties.
\begin{itemize}
    \item[\textup{(i)}] 
        $\inv^{\ast}\sH_{\pi}(\balpha; \bbeta)$ is isomorphic to $\sH_{(-1)^p\pi}(-\bbeta, -\balpha)$,
        where $-\balpha$ \resp{$-\bbeta$} denotes the sequence $-\alpha_1,\ldots,-\alpha_m$ \resp{$-\beta_1,\ldots,-\beta_n$}.
    \item[\textup{(ii)}] Let $\gamma$ be an element of $\bZ_p$. Then, $\sH_{\pi}(\balpha;\bbeta)\otimes^{\dag}_{\sO_{\hbPV,\bQ}(\pdag{\{0,\infty\}})}\sK_{\gamma}$ is isomorphic to
        $\sH_{\pi}(\balpha+\gamma;\bbeta+\gamma)$, where $\balpha+\gamma$ \resp{$\bbeta+\gamma$} denotes the sequence
        $\alpha_1+\gamma,\ldots,\alpha_m+\gamma$ \resp{$\beta_1+\gamma,\ldots,\beta_n+\gamma$}.
\end{itemize}
\end{lemma}

\subsection{$p$-adic Liouville numbers.}

In this subsection, we recall the notion of $p$-adic Liouville numbers and
give a lemma which we need later.

\begin{definition}
    Let $\alpha$ be an element of $\mathbb{Z}_p$.
    We say that $\alpha$ is a \emph{$p$-adic Liouville number} if one of the two power series,
    \[
        \sum_{k\geq 0, k\neq\alpha} \frac{t^k}{\alpha-k} \quad\text{or}\quad
        \sum_{k\geq 0, k\neq-\alpha} \frac{t^k}{\alpha+k}
    \]
    has radius of convergence strictly less than $1$.
\end{definition}

\begin{proposition}[{\cite[13.1.7]{Kedlaya2010}}]
    Let $\alpha$ be an element of $\mathbb{Z}_p\setminus\mathbb{Z}$
    which is not a $p$-adic Liouville number.
    Then, the power series
    \[
        \sum_{k=0}^{\infty}\frac{x^k}{\alpha(1-\alpha)(2-\alpha)\dots(k-\alpha)}
    \]
    has radius of convergence greater than or equal to $p^{-1/(p-1)}$.
    \label{prop:radiusofconvergence}
\end{proposition}

\begin{lemma}
    Let $l$ be a non-negative integer and let $\alpha$ be an element of $\mathbb{Z}_p$.

    \textup{(i)} For any non-negative integer $N\geq l$, the following inequality holds:
    \[
        \left|\prod_{s=l}^N(s-\alpha)\right|\leq p^{-(N-l+1)/(p-1)+1}(N-l+1).
    \]

    \textup{(ii)} Assume that $\alpha$ is neither an integer nor a $p$-adic Liouville number.
    Then, for all positive real number $r$ with $r<p^{-\frac{1}{p-1}}$, we have
    \[
        \lim_{k\to\infty}\left|\prod_{s=l}^{l+k}(s-\alpha)\right|r^{-k}=\infty.
    \]
    \label{lem:evaluation}
\end{lemma}

\begin{proof}
    (i) The proof is the same as that of the first inequality of \cite[3.1.5]{Miyatani}. We include a proof here for the convenience for the reader.

    Since the inequality is trivial if $\alpha\in\{l,\dots,N\}$, we assume that this is not the case.
    For each positive integer $m$, let $t_m$ denote the number of $(s-\alpha)$'s for $s=l,...,N$ that belongs to $p^m\mathbb{Z}_p$:
    \[
        t_m\defeq \#\Set{s\in\{l,\dots,N\} | s-\alpha\in p^m\mathbb{Z}_p}.
    \]
    Then, we have $v_p\left(\prod_{s=l}^N(s-\alpha)\right)=\sum_{m=1}^{\infty}t_m$
    (note that the right-hand side is essentially a finite sum).
    Now, since there is exactly one multiple of $p^m$ in every $p^m$ successive $(s-\alpha)$'s,
    we have $t_m\geq \big\lfloor\frac{N-l+1}{p^m}\big\rfloor$. This shows that
    \[
        v_p\left(\prod_{s=l}^N(s-\alpha)\right)=\sum_{m=1}^{\infty}t_m \geq \sum_{m=1}^{\infty}\left\lfloor\frac{N-l+1}{p^m}\right\rfloor.
    \]
    The right-hand side equals $v_p\big( (N-l+1)!\big)$ and it is well-known that, for any positive integer $M$ we have $v_p(M!)\geq \frac{M}{p-1}-\log_pM-1$.
    Therefore, we have $v_p\left(\prod_{s=l}^N(s-\alpha)\right)\geq \frac{N-l+1}{p-1}-1-\log_p(N-l+1)$,
    from which the assertion follows.

    (ii) Since $l-\alpha$ is neither an integer nor a $p$-adic Liouville number,
    Proposition \ref{prop:radiusofconvergence} shows that the power series
    \[
        \sum_{k=0}^{\infty}\frac{x^k}{(l-\alpha)(l+1-\alpha)\dots(l+k-\alpha)}
    \]
    has radius of convergence greater than or equal to $p^{-\frac{1}{p-1}}$.
    This means that for all $r\in (0,p^{-\frac{1}{p-1}})$, we have
    \[
        \lim_{k\to\infty}\left|\prod_{s=l}^{l+k}(s-\alpha)\right|^{-1}r^k=0,
    \]
    which shows the claim.
\end{proof}

\subsection{A lemma on hypergeometric arithmetic $\sD$-modules under a $p$-adic non-Liouvilleness condition.}
In this subsection, we establish the following lemma that generalizes \cite[Proposition 3.1.4]{Miyatani}.
This lemma plays a central role in proving the main theorem in this article.

\begin{lemma}
    Let $\alpha_1,\dots,\alpha_m$ and $\beta_1,\dots,\beta_n$ be elements of $\mathbb{Z}_p$,
    and assume that $\alpha_i$'s does not have an integer nor have a $p$-adic Liouville numbers.
    Let $j\colon(\hbPV, \{0,\infty\})\to(\hbPV,\{\infty\})$ be the morphism of d-couples defined by $\overline{j}=\id_{\hbPV}$.
    Then, the following assertions hold.
    \begin{itemize}
        \item[\textup{(i)}]
        $j^{\ast}\big(A_1(K)^{\dag}/A_1(K)^{\dag}\Hyp_{\pi}(\balpha;\bbeta)\big)$ is isomorphic to $\sH_{\pi}(\balpha;\bbeta)$.
        \item[\textup{(ii)}]
        The natural morphism
        \[
            A_1(K)^{\dag}/A_1(K)^{\dag}\Hyp_{\pi}(\balpha;\bbeta)\longrightarrow
            j_{+}j^{\ast}\big(A_1(K)^{\dag}/A_1(K)^{\dag}\Hyp_{\pi}(\balpha;\bbeta)\big)
        \]
        is an isomorphism.
    \end{itemize}
    \label{lem:push}
\end{lemma}

\begin{proof}
    (i) follows from the exactness of $j^{\ast}$ on the category of coherent $A_1(K)^{\dag}$-modules.
    The proof of (ii) is, as in the proof of \cite[Proposition 3.1.4]{Miyatani}, reduced to the following Lemma.
\end{proof}

\begin{lemma}
    Let $\alpha_1,\dots,\alpha_m$ and $\beta_1,\dots,\beta_n$ be elements of $\bZ_p$.
    Assume that $\alpha_i$'s does not have an integer nor have a $p$-adic Liouville number.
    Then, on $A_1(K)^{\dag}/A_1(K)^{\dag}\Hyp_{\pi}(\balpha;\bbeta)$,
    the multiplication by $x$ from the left is bijective.
\end{lemma}

\begin{proof}
    \emph{Firstly, we prove the injectivity}.

    To prove this, it suffices to show that
    if $P, Q\in A_1(K)^{\dag}$ satisfy $xP=Q\Hyp_{\pi}(\balpha;\bbeta)$
    then $Q\in xA_1(K)^{\dag}$.
    In fact, then since $x$ is not a zero-divisor in $A_1(K)^{\dag}$,
    we get that $P\in A_1(K)^{\dag}\Hyp_{\pi}(\balpha;\bbeta)$ and the injectivity follows.

    In order to show that $Q\in xA_1(K)^{\dag}$, we may assume that $Q$ is of the form
    $Q=\sum_{l=0}^{\infty}c_l\partial^{[l]}$, where $c_l$'s are elements of $K$ satisfying
    $\exists C>0, \exists \eta<1, \forall l, |c_l|<C\eta^l$.
    Then, by using the congruence $\partial^{[l]}x\equiv\partial^{[l-1]}\pmod{xA_1(K)^{\dag}}$, we have
    \[
        Q\Hyp_{\pi}(\balpha;\bbeta)\equiv\sum_{l=0}^{\infty}c_l\prod_{i=1}^m(l-\alpha_i)\partial^{[l]}-(-1)^{m+np}\pi^{m-n}\sum_{l=1}^{\infty}c_l\prod_{j=1}^n(l-1-\beta_j)\partial^{[l-1]}
    \]
    modulo $xA_1(K)^{\dag}$. By assumption , the left-hand side belongs to $xA_1(K)^{\dag}$, which shows the recurrence relation
    \[
      c_l\prod_{i=1}^m(l-\alpha_i)=(-1)^{m+np}\pi^{m-n}c_{l+1}\prod_{j=1}^n(l-\beta_j).
    \]

    Now, fix a non-negative integer $l$ that exceeds all $\beta_j$'s which are integers.
    Then, by the recurrence relation, we have
    \begin{equation}
        c_{l+k}=(-1)^{k(m+np)}\pi^{-k(m-n)}\frac{\prod_{i=1}^m(l+k-1-\alpha_i)(l+k-2-\alpha_i)\dots(l-\alpha_i)}
        {\prod_{j=1}^n(l+k-1-\beta_j)(l+k-2-\beta_j)\dots(l-\beta_j)}c_l.
        \label{eq:recurrence}
    \end{equation}

    Let us choose $C>0$ and $\eta<1$ such that $\forall l, |c_l|<C\eta^l$.
    The series $\left\{\eta^{-k}|c_{l+k}|\right\}_{k=0}^{\infty}$ is then bounded.

    Now, put $r\defeq\eta^{1/2m}p^{-1/(p-1)}$; if $m=0$, we interpret $\eta^{1/2m}=1$.
    Lemma \ref{lem:evaluation} (i) shows that
    $\left|\dfrac{1}{(l+k-1-\beta_j)\dots(l-\beta_j)}\right|\geq p^{k/(p-1)-1}k^{-1}$ for each $j=1,\dots,n$.
    Moreover, Lemma \ref{lem:evaluation} (ii) shows that
    $\big|(l+k-1-\alpha_i)\dots(l-\alpha_i)\big|r^{-k}\to\infty$ as $k\to\infty$
    for each $i=1,\dots,m$.
    We therefore have, since $|\pi|=p^{-1/(p-1)}$,
    \begin{align*}
        \eta^{-k}|c_{l+k}| &= \eta^{-k}p^{k(m-n)/(p-1)}r^{km}\frac{\prod_{i=1}^m\left\{|(l+k-1-\alpha_i)\dots(l-\alpha_i)|r^{-k}\right\}}{\prod_{j=1}^n|(l+k-1-\beta_j)\dots(l-\beta_j)|}|c_l|\\
        &\geq p^{-n}\left(\eta^{-k/2}k^{-n}\right)\prod_{i=1}^m\left\{\big|(l+k-1-\alpha_i)\dots(l-\alpha_i)\big|r^{-k}\right\}|c_l|.
    \end{align*}
    If $|c_l|\neq 0$, then the right-hand side tends to $\infty$ as $k\to\infty$,
    which contradicts the fact that $\left\{\eta^{-k}|c_{l+k}|\right\}_{k=0}^{\infty}$ is bounded.
    Therefore we have $c_l=0$, and consequently $c_{l+k}=0$ for all $k\geq 0$.
    Now, by the recurrence relation (\ref{eq:recurrence}) and the assumption that $\alpha_i$'s are not integers,
    we get that $Q=0$.

    \emph{Nextly, we prove the surjectivity.}

       Given $P\in A_1(K)^{\dag}$, we have to show that there exists $Q, R\in A_1(K)^{\dag}$
    such that $xQ=P+R\Hyp_{\pi}(\balpha;\bbeta)$.
    To prove this, we may and do assume that $P$ is of the form $P=\sum_{l=0}^{\infty}c_l\partial^{[l]}$,
    where $c_l$'s are elements of $K$ satisfying $\exists C>0, \exists\eta<1, \forall l, |c_l|<C\eta^l$;
    under this assumption, we show that there exists $R\in A_1(K)^{\dag}$ of the form $R=\sum_{d=0}^{\infty}d_l\partial^{[l]}$
    that satisfies $P+R\Hyp_{\pi}(\balpha;\bbeta)\in xA_1(K)^{\dag}$.
    We define a number $l_0$ as follows:
    $l_0$ is the greatest number in $\Set{\beta_j+1 | j\in\{1,\ldots,n\}}\cap\bZ_{\geq 0}$ if
    this set is not empty; 
    we set $l_0=0$ if it is empty.
    
    To prove the existence of $R\in A_1(K)^{\dag}$ as above, we may assume that $c_l=0$ if $l<l_0$ by the following reason.
    If $A_1(K)$ denotes the usual Weyl algebra with coefficients in $K$,
    then since $\alpha_i$'s are not integers,
    the right multiplication by $\Hyp_{\pi}(\balpha;\bbeta)$ is bijective on $A_1(K)/xA_1(K)$ 
    \cite[2.9.4, (3)$\Rightarrow$(2)]{Katz90ESDE}.
    This shows that there exists $R'\in A_1(K)$ such that
    $\sum_{l=0}^{l_0-1}c_l\partial^{[l]}+R'\Hyp_{\pi}(\balpha;\bbeta)\in xA_1(K)$
    (The proof in the reference \cite{Katz90ESDE} is given over $\bC$, but it remains valid for all field of characteristic $0$).
    Now, we assume that $c_l=0$ if $l<l_0$.

    We put $d_l=0$ if $l<l_0$, and for each $s\geq 0$ we put
    \begin{equation}
        d_{l_0+s} = \sum_{t=s}^{\infty}(-1)^{(t-s)(m+np+1)}\pi^{(t-s)(m-n)}\frac{\prod_{j=1}^n(l_0+t-1-\beta_j)\ldots(l_0+s-\beta_j)}{\prod_{i=1}^m(l_0+t-\alpha_i)\ldots(l_0+s-\alpha_i)}c_{l_0+t};
        \label{eq:defofd}
    \end{equation}
    let us firstly check that this infinite series actually converges.
    Lemma \ref{lem:evaluation} (i) shows that $\big|(l_0+t-1-\beta_j)\dots(l_0+s-\beta_j)\big|\leq p^{-(t-s)/(p-1)+1}(t-s)$.
    Let $C>0$ and $\eta<1$ be numbers such that $\forall l, |c_l|<C\eta^l$, and put $r\defeq\eta^{1/2m}p^{-1/(p-1)}$
    (as before, if $m=0$, then we interpret $\eta^{1/2m}=1$).
     Then, Lemma \ref{lem:evaluation} (ii) shows that
    $\dfrac{1}{\big|(l_0+t-\alpha_i)\dots(l_0+s-\alpha_i)\big|}r^{t-s}\to 0$ as $t\to\infty$.
    Therefore, the norm of each summand in the right-hand side of (\ref{eq:defofd}) is bounded from above by
    \begin{align*}
        & p^{-(t-s)m/(p-1)+n}(t-s)^nr^{-(t-s)m}\prod_{i=1}^m\left\{\frac{1}{\big|(l_0+t-\alpha_i)\dots(l_0+s-\alpha_i)\big|}r^{t-s}\right\}C\eta^{l_0+t}\\
        \leq & Cp^n\left\{(t-s)^n\eta^{l_0+(s+t)/2}\right\}\prod_{i=1}^m\left\{\frac{1}{\big|(l_0+t-\alpha_i)\dots(l_0+s-\alpha_i)\big|}r^{t-s}\right\},
    \end{align*}
    and the right-hand side converges to $0$ as $t\to\infty$.
    We have now checked that the right-hand side of (\ref{eq:defofd}) converges and that thus $d_{l_0+s}$ is well-defined.

    Nextly, we put $R\defeq\sum_{l=0}^{\infty}d_l\partial^{[l]}$ and prove that $R\in A_1(K)^{\dag}$.
    By the bound of the each summand of (\ref{eq:defofd}) given above, we have
    \begin{equation}
        |d_{l_0+s}|< Cp^n\max_{t\geq s}\left[\left\{(t-s)^n\eta^{l_0+(s+t)/2}\right\}\prod_{i=1}^m\left\{\frac{1}{\big|(l_0+t-\alpha_i)\dots(l_0+s-\alpha_i)\big|}r^{t-s}\right\}\right].
        \label{eq:evalofd}
    \end{equation}
    If $m=0$, then it is easy to check that there exists a constant $C'>0$
    such that $|d_{l_0+s}|\leq C'\eta^{s/2}$. We thus assume that $m>0$.

    For each $i=1,\dots,m$, Lemma \ref{lem:evaluation} (i) shows the inequality
    \begin{align*}
        \frac{1}{\big|(l_0+t-\alpha_i)\dots(l_0+s-\alpha_i)\big|}r^{t-s}
        &= \frac{\big|(l_0+s-1-\alpha_i)\dots(l_0-\alpha_i)\big|}{\big|(l_0+t-\alpha_i)\dots(l_0-\alpha_i)\big|}r^{t-s}\\
        &\leq \frac{1}{\big|(l_0+t-\alpha_i)\dots(l_0-\alpha_i)\big|}r^{t-s}p^{-s/(p-1)+1}s\\
        &= \frac{r^t}{\big|(l_0+t-\alpha_i)\dots(l_0-\alpha_i)\big|}ps\eta^{-s/2m}.
    \end{align*}
    By Lemma \ref{lem:evaluation} (ii) , the fraction $\dfrac{r^t}{\big|(l_0+t-\alpha_i)\dots(l_0-\alpha_i)\big|}$
    is bounded by a constant independent of $t$.
     Therefore, by looking at (\ref{eq:evalofd}), there exists a constant $C_1>0$ such that
    \begin{align*}
        |d_{l_0+s}|&<C_1\max_{t\geq s}\left\{(t-s)^n\eta^{(s+t)/2}\right\}s^n\eta^{-s/2}\\
        &= C_1\max_{t\geq s}\left\{(t-s)^n\eta^{(t-s)/2}\right\}s^n\eta^{s/2}\\
        &= C_1\max_{t\geq 0}\left\{t^n\eta^{t/2}\right\}\left(s^n\eta^{s/4}\right)\eta^{s/4}.
    \end{align*}
    Now, $C_1\max_{t\geq 0}\left\{t^n\eta^{t/2}\right\}\left(s^n\eta^{s/4}\right)$ is bounded by a constant $C_2$ independent of $s$
    and we have $|d_{l_0+s}|\leq C_2\eta^{s/4}$ for all $s\geq 0$.
    This proves that $R=\sum_{l=0}^{\infty}d_l\partial^{[l]}$ belongs to $A_1(K)^{\dag}$.

    It remains to prove that $R$ satisfies $P+R\Hyp_{\pi}(\balpha;\bbeta)\in xA_1(K)^{\dag}$,
    and this is just a formal calculation.
    In fact, it is equivalent to showing that
    \[
        d_l\prod_{i=1}^m(l-\alpha_i)-(-1)^{(m+np)}\pi^{m-n}d_{l+1}\prod_{j=1}^n(l-\beta_j)+c_l=0
    \]
    for all $l\geq 0$.
    It trivially holds if $l<l_0-1$ because $d_l=d_{l+1}=c_l=0$ in this case;
    it also holds if $l=l_0-1$ because $d_l=c_l=0$ and $l-\beta_j=0$ for some $j$;
    otherwise, we may check it directly by using (\ref{eq:defofd}).
\end{proof}

\section{Hypergeometric Arithmetic $\sD$-modules and Multiplicative Convolution.}

\subsection{Main Theorem.}
Now, we are ready to state and prove the main theorem of this article.

\begin{theorem}
  \label{thm:maintheorem}
  Let $\balpha=(\alpha_1,\dots,\alpha_m)$ and $\bbeta=(\beta_1,\dots,\beta_n)$ be
  sequences of elements of $\mathbb{Z}_p$.
  Assume that, for any $i$ and $j$, $\alpha_i-\beta_j$ is not an integer
  nor a $p$-adic Liouville number.
  
  \textup{(i)} Assume that $m\geq 1$ and put $\balpha'=(\alpha_2,\dots,\alpha_m)$.
  Then, we have an isomorphism
    \[
      \sH_{\pi}(\balpha'; \bbeta)\ast\sH_{\pi}(\alpha_1; \emptyset)[-1]\cong\sH_{\pi}(\balpha;\bbeta).
    \]

  \textup{(ii)} Assume that $n\geq 1$ and put $\bbeta'=(\beta_1,\dots,\beta_n)$.
  Then, we have an isomorphism
    \[
      \sH_{\pi}(\balpha; \bbeta')\ast\sH_{\pi}(\emptyset; \beta_1)[-1]\cong\sH_{\pi}(\balpha;\bbeta).
    \]
\end{theorem}

\begin{proof}
    We prove (i) and (ii) by induction on $m+n$.
    If $(m,n)=(1,0)$ \resp{$(m,n)=(0,1)$}, then (i) \resp{(ii)} follows from the fact
    that $\sH_{\pi}(\emptyset;\emptyset)$ is a unit object for the multiplicative convolution.
    The latter fact can be checked as in the proof of \cite[2.1.2]{Miyatani}.

    Now, assume that $m+n\geq 2$ and let us prove the assertions (i) and (ii).
    In fact, Lemma \ref{lem:calcofhyp} (i)
    and the isomorphism $\inv^{\ast}(\sM\ast\sN)\cong(\inv^{\ast}\sM)\ast(\inv^{\ast}\sN)$,
    whose proof is straightforward and left to the reader,
    show that (ii) is deduced from (i).

    The proof of (i) is reduced to the case where $\alpha_1=0$ as follows.
    Because of the isomorphism $\lambda'^!\sK_{\alpha_1}\cong f'^!\left(\pr_1^!\sK_{\alpha_1}\totimes\pr_2^!\sK_{\alpha_1}\right)[1]$,
    we have
    \begin{align*}
        \big(\sM\ast\sH_{\pi}(0,\emptyset)\big)\totimes\sK_{\alpha_1}
        & =
        \lambda'_{+}f'^!\big(\pr_1^!\sM\totimes\pr_2^!\sH_{\pi}(0,\emptyset)\big)\totimes\sK_{\alpha_1} \\
        & \cong
        \lambda'_{+}\left(f'^!\big(\pr_1^!\sM\totimes\pr_2^!\sH_{\pi}(0,\emptyset)\big)\totimes\lambda'^!\sK_{\alpha_1}\right) \\
        & \cong
        \lambda'_{+}f'^!\left(\pr_1^!\sM\totimes\pr_2^!\sH_{\pi}(0,\emptyset)\totimes\pr_1^!\sK_{\alpha_1}\totimes\pr_2^!\sK_{\alpha_1}\right)[1]\\
        & \cong
        \lambda'_{+}f'^!\left((\sM\totimes\sK_{\alpha_1})\Ldagboxtimes\big(\sH_{\pi}(0;\emptyset)\totimes\sK_{\alpha_1}\big)\right)[1]\\
        & \cong
        \lambda'_{+}f'^!\left((\sM\totimes\sK_{\alpha_1})\Ldagboxtimes\sH_{\pi}(\alpha_1;\emptyset)\right)\\
        & \cong
        (\sM\totimes\sK_{\alpha_1})\ast\sH_\pi(\alpha_1;\emptyset).
    \end{align*}
    Therefore, if the assertion (i) is proved for $\alpha_1=0$,
    then we get the desired theorem for general $\alpha_1$ by tensoring $\sK_{\alpha_1}$,
    with the aid of Lemma \ref{lem:calcofhyp} (ii).

    In the case where $\alpha_1=0$,
    we may prove the assertion in the same way as \cite[Theorem 3.2.5]{Miyatani}.
    We include here a sketch of the proof.

    By the induction hypothesis, $\sH_{\pi}(\balpha';\bbeta)\cong\sH_{\pi}(\balpha'+1;\bbeta+1)$
    because for the Kummer isocrystals $\sK_{\gamma}$ we have an isomorphism $\sK_{\gamma}\cong\sK_{\gamma+1}$.
    Therefore, since $(-\beta_j-1)$'s do not have a $p$-adic Liouville number,
    we see by Lemma \ref{lem:calcofhyp} (i) and Lemma \ref{lem:push} that
    \begin{align*}
        j_{+}\inv^{\ast}\sH_{\pi}(\balpha';\bbeta) &\cong j_+\inv^{\ast}\sH_{\pi}(\balpha'+1;\bbeta+1) \\
        & \cong j_+\sH_{(-1)^p\pi}(-\bbeta-1;-\balpha'-1)\\
        & \cong A_1(K)^{\dag}/A_1(K)^{\dag}\Hyp_{(-1)^p\pi}(-\bbeta-1;-\balpha'-1).
    \end{align*}
    Because this is a coherent $A_1(K)^{\dag}$-module,
    Proposition \ref{prop:fourierandconv} shows that
    \[
        \sH_{\pi}(\balpha';\bbeta)\ast\left(j^{\ast}\sL_{\pi}[-1]\right)
        \cong j^{\ast}\big(\FT_{\pi}\big(j_{+}\inv^{\ast}\sH_{\pi}(\balpha';\bbeta)\big)\big).     
    \]
    Finally, by a direct calculation using Proposition \ref{prop:arithfourier},
    we may prove the isomorphism
    \[
      \FT_{\pi}\left(A_1(K)^{\dag}/A_1(K)^{\dag}\Hyp_{(-1)^p\pi}(-\bbeta-1;-\balpha'-1)\right)
      \cong A_1(K)^{\dag}/A_1(K)^{\dag}\Hyp_{\pi}(\balpha;\bbeta)
    \]
    (cf. the proof of \cite[3.2.5]{Miyatani}). Now the assertion follows by Lemma \ref{lem:push} (i).
\end{proof}

\subsection{Quasi-$\Sigma$-unipotence.}

In this last subsection, 
we discuss the quasi-$\Sigma$-unipotence of arithmetic hypergeometric $\mathscr{D}$-modules.

\begin{para}
    Let $\Sigma$ be the subgroup of $\mathbb{Z}_p/\mathbb{Z}$
    that does not contain a $p$-adic Liouville number.
    Caro \cite[3.3.5]{Caro18MM} defines, for each smooth formal scheme $\sP$ over $V$,
    the subcategory $D^{\rb}_{\rqq\hyphen\Sigma}\big(\sD^{\dag}_{\sP,\bQ}\big)$
    of $D^{\rb}_{\coh}\big(\sD^{\dag}_{\sP,\bQ}\big)$
    consisting of ``quasi-$\Sigma$-unipotent'' objects.
    These categories are stable under Grothendieck's six operations.
\end{para}

\begin{proposition}
    \label{prop:overholonomic}
    Let $\balpha=(\alpha_1,\dots,\alpha_m)$ and $\bbeta=(\beta_1,\dots,\beta_n)$ be sequences of elements of $\mathbb{Z}_p$,
    and assume that $(m,n)\neq (0,0)$.
    Let $\Sigma$ be the subgroup of $\mathbb{Z}_p/\mathbb{Z}$
generated by the canonical images of $\alpha_i$'s and $\beta_j$'s.
Assume that $\alpha_i-\beta_j\not\in\mathbb{Z}$ for any $i, j$, 
and that $\Sigma$ does not contain the canonical image of a $p$-adic Liouville number.

Then, $\sH_{\pi}(\balpha;\bbeta)$ is an object of $D^{\rb}_{\rqq\hyphen\Sigma}\big(\sD^{\dag}_{\hbPV,\bQ}\big)$.
In particular, it is an overholonomic $\sD^{\dag}_{\hbPV,\bQ}$-module.
\end{proposition}
\begin{proof}
    If the canonical image of a $p$-adic number $\gamma\in\mathbb{Z}_p$ in $\mathbb{Z}_p/\mathbb{Z}$
    belongs to $\Sigma$,
    then $\sK_{\gamma}$ is an object of $D^{\rb}_{\rqq\hyphen\Sigma}\big(\sD^{\dag}_{\hbPV,\bQ}\big)$
    because it is (the realization on $(\mathbb{G}_{\rmm,k}, \hbPV)$ of)
    an overconvergent isocrystal on $\mathbb{G}_{\rmm,k}$ whose exponent is $\gamma\in\Sigma$ (resp. $-\gamma\in\Sigma$) at $0$ (resp. at $\infty$),
    and because $D^{\rb}_{\rqq\hyphen\Sigma}\big(\sD^{\dag}_{\hbPV,\bQ}\big)$ contains all such objects by construction.

    We may also show that $\sL_{\pi}$ is also an object of $D^{\rb}_{\rqq\hyphen\Sigma}\big(\sD^{\dag}_{\hbPV,\bQ}\big)$.
    In fact, it is a direct factor of the push-forward of the trivial isocrystal on $\mathbb{A}^1_{k}$
    along the Artin--Schreier morphism.
    Now, the trivial isocrystal on $\mathbb{A}^1_{k}$ is an object
    of $D^{\rb}_{\rqq\hyphen\Sigma}\big(\sD^{\dag}_{\hbPV,\bQ}\big)$ (the exponent at $\infty$ is $0\in\Sigma$).
    Since $D^{\rb}_{\rqq\hyphen\Sigma}$ is stable under push-forward and direct factor,
    the claim follows.

    Now, by Remark \ref{rem:m+n=1}, the corollary holds for $(m,n)=(1,0), (0,1)$.
    For general $(m,n)$, Theorem \ref{thm:maintheorem} and the stability
    of $D^{\rb}_{\rqq\hyphen\Sigma}$ under Grothendieck's six functors show the assertion.
\end{proof}

\bibliographystyle{amsalpha}
\bibliography{references}

\end{document}